\documentclass[11pt,a4paper,reqno]{amsart}
\usepackage[english]{babel}
\usepackage{pst-grad} % For gradients
\usepackage{pst-plot} % For axes
\usepackage{pstricks}
\usepackage{amsmath,amssymb,mathrsfs,amsthm, mathtools}
\usepackage{tikz} 
\usepackage{mathcomp,wasysym}  
\usepackage{graphicx}  
\usepackage[all]{xy} \xyoption{arc} \xyoption{color}
\usepackage{epsfig}
\usepackage{cite}
\usepackage[a4paper,
left=2.5cm, right=2.5cm,
top=3cm, bottom=3cm]{geometry}
\newcommand{\dx}{\textnormal{d}{x}}

\renewcommand{\div}{\textnormal{div}}
\newcommand{\aux}{\textnormal{aux}}
\newcommand{\NS}{Navier-Stokes}

\newcommand{\pare}[1]{\left( #1 \right)}

\newcommand{\av}[1]{\left| #1 \right|}
\newcommand{\bra}[1]{\left[ #1 \right]}
\newcommand{\set}[1]{\left\{ #1 \right\}}
\newcommand{\sgn}{\textnormal{sgn}}

\newcommand{\cH}{\mathcal{H}}
\newcommand{\cC}{\mathcal{C}}

\newcommand{\bR}{\mathbb{R}}
\newcommand{\bN}{\mathbb{N}}

\newcommand{\cG}{\mathcal{G}}

\newcommand{\cO}{\mathcal{O}}
\newcommand{\cU}{\mathcal{U}}
\newcommand{\cK}{\mathcal{K}}
\newcommand{\cN}{\mathcal{N}}
\newcommand{\cT}{\mathcal{T}}

\newcommand{\pat}{\partial_t}
\newcommand{\pax}{\partial_x}
\newcommand{\pa}{\partial}

\newcommand{\vertiii}[1]{{\left\vert\kern-0.25ex\left\vert\kern-0.25ex\left\vert #1 
    \right\vert\kern-0.25ex\right\vert\kern-0.25ex\right\vert}}

\normalsize
\normalsize
\usepackage[bookmarks,colorlinks]{hyperref}

\newtheorem{satz}{Proposition}[section]
 
\newtheorem{remark}[satz]{Remark}

\newtheorem{lemma}[satz]{Lemma}

\title{Asymptotic models for free boundary flow in porous media}

\author[R. Granero-Belinch\'{o}n]{Rafael Granero-Belinch\'{o}n}
\email{rafael.granero@unican.es}
\address{Departamento  de  Matem\'aticas,  Estad\'istica  y  Computaci\'on,  Universidad  de Cantabria.  Avda.  Los  Castros  s/n,  Santander,  Spain.}

\author[S. Scrobogna]{Stefano Scrobogna}
\email{sscrobogna@bcamath.org}
\address{Basque Center for Applied Mathematics, Mazarredo 14, 48009, Bilbao, Basque Country, Spain}

\begin{document}
\begin{abstract}
We provide rigorous asymptotic models for the free boundary Darcy and Forchheimer problem under the assumption of weak nonlinear interaction, in a regime in which the steepness parameter of the interface is considered to be very small. The models we derive capture the nonlinear interaction of the original free boundary Darcy and Forchheimer problem up to quadratic terms. Furthermore, we provide models that consider both the two-dimensional and three-dimensional cases, with and without bottom topography.
\end{abstract}

%\subjclass{35R35, 35Q35, 35S10, 76B03}
\keywords{Muskat problem, Darcy law, Forchheimer flow, moving interfaces, free-boundary problems}

%\ccode{2010 AMS Subject Classification: 35455, 35B41, 92C17}

\maketitle
{\small
\tableofcontents}

\allowdisplaybreaks
\section{Introduction}
Flow in porous media is important in many different applications ranging from oil production to catalytic converters. The simplest equation modeling flow in porous media is known as Darcy's law and reads 
\begin{equation}
\frac{\mu}{\kappa} u=-\nabla p-\rho G e_2,
\label{eq:Darcy}
\end{equation}
where $u$, $p$, $\rho$ and $\mu$ are the velocity, pressure, density and dynamic viscosity of the fluid, respectively. The constant $\kappa$ describes a property of the porous media and its known as the permeability. $G e_2$ stands for the acceleration due to gravity in the direction $(0,1)^\intercal $. Darcy law is valid for slow and viscous flows, and it was {first} derived experimentally by {Henry Darcy}  {in} 1856 and then derived theoretically from the \NS\ equations via homogenization (cf. \cite{Whitaker1986}). Darcy law is widely used in applications. In particular, the free boundary Darcy flow,  {also known as the Muskat problem} (cf. \cite{Muskat31, Muskat32, Muskat34}), appears as a model of geothermal reservoirs  \cite{CF}, aquifers or oil wells \cite{Muskat:porous-media}. Remarkably, the Muskat problem is mathematically analogous to the Hele-Shaw cell problem (see \cite{HS1898, ST58,CP93, cheng2012global}) that studies the movement of a fluid trapped between two parallel vertical plates, which are separated by a very narrow distance. Despite the Muskat problem has a long history in the physical literature, the rigorous mathematical analysis of the equation \eqref{eq:Darcy} with free boundary is relatively recent (we refer the interested reader to \cite{c-g07, CCGS13, GG, CGS16, CL18} and the references therein).\\

%Indeed if the region below the interface has higher density the nonlinear equation \eqref{eq:Darcy} is linearly stable in $ H^2 $ (\cite{CGS16}), when the initial data is Lipschitz continuous (\cite{CCGS13}) and very recently it has been proved that it is linearly stable in $ H^{3/2} $ (we refer the reader to \cite{CL18}).  \\

When the Reynolds number of the flow becomes larger, inertial terms should be added into the conservation of momentum equation. For these high velocity flows, Forchheimer \cite{Forchheimer14} noted that 
\begin{equation}\label{eq:forchheimer}
\beta\rho |u|u+\frac{\mu}{\kappa} u=-\nabla p-\rho G e_2,
\end{equation}
is a more accurate conservation of momentum equation. Here $\beta$ is known as the Forchheimer coefficient and the term $ \beta\rho |u|u $ amounts to inertial effects of the flow. \\

The scope of the present paper is to provide simplified models which approximate the evolution of the free boundary Darcy and free boundary Forchheimer problems under an assumption of weak nonlinearity (see equations \eqref{eqfapprox} and \eqref{eqf2Forch} below). We choose to consider hence a configuration in which the interface is not very steep. More explicitly, if we denote by $ H $ and $ L $ respectively the typical amplitude and wavelength of the interface and we consider the steepness parameter $ \sigma = H/L $, we suppose that $ 0<\sigma \ll 1 $.  Such configuration is rather common in geophysical fluid dynamics and it has been widely used in order to derive asymptotic expansions for the water wave problem (we refer the reader to the classical work of Stokes \cite{Stokes_1847} and to the more recent works \cite{AN10, AN12, AN14, NR05, NR06}). In such a setting we derive asymptotic models for the free boundary Darcy and the free boundary Forchhimer problems which capture the nonlinear interactions of \eqref{eq:Darcy} and \eqref{eq:forchheimer} up to quadratic terms. \\

In the first part of the paper, as a starting point, we consider the free boundary Darcy problem when the depth is infinite and the dimension of the interface is one. We observe that these assumptions on the dimension of the interface and the depth are not really necessary and will be removed below (see section \ref{sec:multidimensional_Darcy}). Starting with the Darcy equation in a moving domain (we refer the reader to the Cauchy problem \eqref{muskat} for a full presentation of the equations considered), we nondimensionalize the equation of motion redefining appropriate dimensionless unknowns and variables. Such nondimensionalization allows us to make appear explicitly the \textit{steepness parameter} $ \sigma = H/L $ in the equations of motion. We can next reformulate the problem, which is defined at the moment on a time-dependent domain $ \Omega\pare{t} $, on a fixed domain $ \Omega $; this is done through a diffeomorfic change of variables. Similar ideas were used previously in the study of nonlinear PDEs with moving domain. For instance, we refer to the works of Matsuno \cite{matsuno1992nonlinear, matsuno1993two, matsuno1993nonlinear}, Granero \& Shkoller \cite{GS}, Cheng, Granero, Shkoller \& Wilkening \cite{CGSW18}, Coutand \& Shkoller \cite{coutand2014finite, Coutand-Shkoller:well-posedness-free-surface-incompressible} and Lannes \cite{Lannes13} for the water waves and Rayleigh-Taylor instability problem. At this point we suppose that the ensemble of the unknowns of the problem, which we denote at the moment as $ \cU $ for the sake of brevity, can be expressed a series of powers of $ \sigma $, i.e.
\begin{equation}\label{eq:stokes_expansion}
\cU \pare{x, t} = \sum_{k\geqslant 0} \cU^{\pare{k}} \pare{x, t} \ \sigma^k. 
\end{equation}
At this point we can simply drop every $ \cO\pare{\sigma^3} $ term in the sequence of systems derived and what remains is the first- and second-order approximation of the Muskat problem in terms of the steepness parameter $ \sigma $. Next a technical result is proved (see Lemma \ref{lem:p1X}) which is inspired by the very recent work \cite[Lemma 1]{CGSW18} which allows us to express the approximation of the evolution of the Muskat problem as an evolution problem on the boundary. With this method we derive equation \eqref{eqf}. The first advantages of the technique introduced above is that it only requires elementary mathematical tools. Another advantage is that it can be easily adapted to also handle the case of Forchheimer flow.\\

% {Remarkably, an expansion as \eqref{eq:stokes_expansion} can be used to prove existence of solutions for certain PDEs (for instance, the interested reader can refer to the works by Oseen and Knightly for the incompressible \NS\ equations \cite{oseen1911formules, Kn1966, Lemarie16})}.

Then we use the previous procedure to obtain a new asymptotic model for the Forchheimer equation \eqref{eq:forchheimer} with moving boundary when the depth is assumed to be infinite and the dimension of the interface is one. In this way we derive equation \eqref{eq:asymptotic_Forchheimer}.\\

Finally, in Sections \ref{sec:multidimensional_Darcy} and \ref{ref:asymDarcy3D} we extend our results for the free boundary Darcy problem and provide an asymptotic model for the free-surface Darcy flow in two and three space dimensions (\emph{i.e.} when the dimension of the interface is one or two) and with or without flat bottom. Although the previous method can be squezzed to handle bounded three dimensional fluid domains, we will use a different technique. We take advantage of the irrotationality of the flow in order to write the equations in terms of the velocity potential. Such potential solves an elliptic equation (see \eqref{eq:elliptic_equation_Phi}), hence it can be completely determined by its trace on the interface, which is a function of the elevation $ h $; in such a way we manage to write the evolution of $ h $ as
\begin{equation*}
\partial_t h = \cN\bra{h}, 
\end{equation*}
where $ \cN $ is a nonlinear function of $ h $. Next we expand $ \cN $ in terms of the steepness parameter and we obtain the asymptotic models \eqref{eq:multid} and \eqref{eq:multid3}. This is a very versatile method that requires a solid knowledge of elliptic theory and other mathematical tools such as the \textit{Dirichlet--Neumann operator} (cf. \cite[Chapter 3]{Lannes13,CL14, CL15}).

The rigorous mathematical analysis of the derived asymptotic equation \eqref{eqf} for the Muskat problem is performed in the forthcoming paper \cite{GBS_analysis}.

\subsection{Plan of the paper}
For the sake of clarity we first consider a fluid moving according to Darcy law when the depth is infinite and the flow is two-dimensional (one-dimensional interface). Then, in section \ref{ref:Darcy}, we introduce the Eulerian form of the problem along with its non-dimensionalization and its Arbitrary Lagrangian-Eulerian formulation. Later on, in section \ref{ref:asymDarcy}, we obtain the first of our asymptotic models for free boundary flow in porous media. Once we have introduced the main ideas of the paper in the simpler setting of Darcy law, we turn our attention to the more nonlinear Forchheimer flow in section \ref{ref:Forchheimer}. In this section we introduce the Eulerian formulation, the non-dimensionalization and the Arbitrary Lagrangian-Eulerian set of equations for Forchheimer flow. In Section \ref{sec:forchheimer_fixed domain} we derive our asymptotic model for the Forchheimer flow. Finally, in Sections \ref{sec:multidimensional_Darcy} and \ref{ref:asymDarcy3D} we provide a multidimensional asymptotic model for the Darcy flow with finite depth and a (possibly) flat bottom when the flow is three dimensional (two dimensional interface).

\subsection{Notations and conventions}

\subsubsection{Matrix indexing} Let $A$ be a matrix, and $b$ be a column vector. Then, we write $A^i_j$ for the component of $A$, located on row $i$ and column $j$;  consequently, using the Einstein summation convention, we write
$$
(Ab)^k=A^k_ib^i\text{ and }(A^Tb)^k=A^i_k b^i.
$$

%\subsubsection{Series summation}
%We adopt the convention that independent of the summand $s_j$, 
%\begin{equation}\label{summation}
%\sum_{j=0}^{k-\ell} s_j=0 \ \ \text{ whenever } k<\ell \,.
%\end{equation} 

\subsubsection{Derivatives}
We write
$$
\partial_j f =\frac{\partial f}{\partial x_j},\quad \partial_t f=\frac{\partial f}{\partial t }
$$
for the space derivative in the $j-$th direction and for a time derivative, respectively. When two spatial variables are considered, we write
$$
\nabla^\perp=\left(-\partial_2,\partial_1\right).
$$
%\textcolor{magenta}{When the fluid fill a three dimensional domain, we also use the notation $ X = \pare{x_1, x_2} \in \bR^2$, and let us denote $ \pare{X, z} $ any element of $ \bR^{3} $. We will refer to $ X $ as to the "horizontal" variables and to $ z $ as the "vertical variable". Indeed $ \nabla_{X, z} = \pare{\nabla_X, \partial_z} = \pare{\nabla, z} $.}

\subsubsection{Fourier series and singular integral operators}   {Let $f(x_1)$ denote a $L^2$ function on $\mathbb{S}^1$ (identified with the interval $[-\pi,\pi]$ with periodic boundary conditions. Then it has the following Fourier representation
$f(x_1) = \sum\limits_{k=- \infty }^ \infty  \hat{f}(k) \ e^{ikx_1}$ for all $x_1 \in \mathbb{S}^1$, where 
$$
\hat f(k) = {\dfrac{1}{2\pi}} \int_{\mathbb{S}^1}{}\, f(x_1) \ 
e^{-ikx_1}dx_1.
$$} 
Using the Fourier representation, we define the
Hilbert transform $\mathcal{H}$ and the Calderon operator $ \Lambda $,  respectively,  as
\begin{align}\label{Hilbert}
\widehat{\mathcal{H}f}(k)=-i\text{sgn}(k) \hat{f}(k) \,, \ \ 
\widehat{\Lambda f}(k)&=|k|\hat{f}(k)\,.
\end{align}
%In particular, we note that
%$$
%\partial_1 \mathcal{H}=\Lambda,\quad \mathcal{H}^2=-1.
%$$

\section{Two dimensional Darcy flow}\label{ref:Darcy}
\subsection{The fluid domain}\label{ref:domain}
The time-dependent two-dimensional infinitely deep fluid domain and free boundary are defined as
\begin{align}\label{Omegat}
\Omega(t) & = \set{ (x_1, x_2)\in\bR^2 \ \Big| \ {-L}\pi <x_1<{L}\pi\,, -\infty < x_2 < h(x_1,t)\,, \ t\in[0,T] }, \\
%----------------------------------
\label{Gammat}
\Gamma(t) & = \set{ \pare{x_1,h(x_1,t)}\in\bR^2 \ \Big|  {-L}\pi <x_1<{L}\pi\,,\ t\in[0,T] }
\end{align} 
with periodic boundary conditions in the horizontal variable $x_1$. We note that $L$ is related to the typical wavelength of the wave.
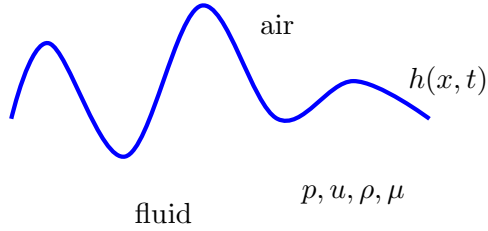
\begin{figure}[h]\label{fig1}
\begin{center}
\begin{tikzpicture}[scale=0.5]
    \draw (17,2.5) node { air};
    \draw (21.5,1.) node { $h(x,t)$}; 
    \draw (14,-2.5) node { fluid};
    \draw (19,-2) node { $p,u, \rho,\mu$}; 
    \draw[color=blue,ultra thick] plot[smooth,tension=.6] coordinates{( 10,0) (11,2) (13,-1) (15, 3) (17, 0) (19, 1) (21,0) };
    \end{tikzpicture} 
\end{center}
\caption{The fluid-air interface $h(x,t)$.}
\end{figure}
We define the reference domain $\Omega$ and reference interface $\Gamma$ as
\begin{align}\label{Omega}
\Omega = \mathbb{S}^1 \times (-\infty, 0) \,, &&
\Gamma = \mathbb{S}^1 \times \{0\} \,.
\end{align} 
We let $N =e_2$ denote the outward unit normal to $\Omega$ at $\Gamma$, and we let $\tau(x_1 ,t)$ and $n( x_1 , t)$ denote, respectively, the unit tangent and (outward) normal vectors to $\Gamma(t)$
\begin{align*}
\tau=\frac{(1,\partial_1 h)}{\sqrt{1+(\partial_1 h)^2}}, && n=\frac{(-\partial_1 h,1)}{\sqrt{1+(\partial_1 h)^2}}.
\end{align*}
The induced metric for $\Gamma(t)$ is given by
\begin{equation}\label{metric}
g =  1+ (\partial_1 h)^2\,.
\end{equation}

\subsection{The equations in the Eulerian formulation}
Slow, viscous flow in two-dimensional porous media can be modelled with the following set of equations (known also as the one-phase Muskat problem):
\begin{subequations}\label{muskat}
\begin{alignat}{2}
\frac{\mu}{\kappa} u+\nabla p&=-\rho G e_2,  \qquad&&\text{in}\quad \Omega(t)\times[0,T]\,,\\
\nabla\cdot u &=0,  \qquad&&\text{in}\quad \Omega(t)\times[0,T]\,,\\
p &= -\gamma \mathcal{K}_{\Gamma(t)}\qquad &&\text{on }\Gamma(t)\times[0,T],\\
\pat h &= u\cdot(-\pax h,1)\qquad &&\text{on }\Gamma(t)\times[0,T],
\end{alignat}
\end{subequations}
where $u$ (units of $length/time$) and $p$ (units of $mass/time^2$) are the velocity and pressure of the fluid. The constants $\mu$ (units of $mass/(length\cdot time)$) and $\rho$ (units of $mass/length^2$) denote the dynamic viscosity and density of the fluid. The constants $\kappa$ (units of $length$) and $G$ (units of $length/time^2$) denote the permeability of the porous media and the gravity, respectively. Moreover,$\gamma$ is the surface tension coefficient (units of $mass\cdot length/time^2$) at the interface, while $\mathcal{K}_{\Gamma(t)}$ denotes the curvature of the interface
$$
\mathcal{K}_{\Gamma(t)}=\frac{\partial_1^2 h}{\left(1+\pare{ \partial_1 h}^2\right)^{3/2}}.
$$

The system (\ref{muskat}) is supplemented with an initial condition for $h$:
\begin{equation}\label{eq:initial}
h(0,x)=h_0(x)
\end{equation}

Instead of using the formulation in terms of the Eulerian velocity and pressure, \eqref{muskat} can be formulated in terms of the stream function and the tangential velocity (see \cite{CGSW18} for the analog situation for water waves). Indeed, define the tangential velocity (or vorticity strength)
$$
\omega = -u \cdot \tau    \text{ on } \ \Gamma(t) \,,
$$
and
$$
\nabla^\perp\psi = u \text{ in } \ \Omega(t) \,.
$$
Then, we observe that
$$
\omega= -\nabla^\perp \psi\cdot \tau=\nabla \psi\cdot n \text{ on } \ \Gamma(t),
$$
$$
\partial_t h= \nabla^\perp \psi\cdot n=\nabla \psi\cdot \tau=\partial_1\left(\psi|_{\Gamma(t)}\right) \text{ on } \ \Gamma(t),
$$
We also compute that, 
\begin{align*}
 {\frac{\mu}{\kappa}}\sqrt{g}\omega&=-{\frac{\mu}{\kappa}} u\cdot \sqrt{g}\tau\\
&=\nabla p\bigg{|}_{\Gamma(t)}\cdot\sqrt{g}\tau +\rho G \partial_1 h\\
&=\partial_1\left( p|_{\Gamma(t)}\right) +\rho G \partial_1 h\\
&=\partial_1 \left(-\gamma\frac{\partial_1^2 h}{\left(1+(\partial_1 h)^2\right)^{3/2}}\right) +\rho G \partial_1 h
\end{align*}

Then, we have that \eqref{muskat} is equivalent to
\begin{subequations}\label{muskat2}
\begin{alignat}{2}
\Delta \psi&=0,  \qquad&&\text{in}\quad \Omega(t)\times[0,T]\,,\\
\nabla \psi \cdot n&=\frac{\kappa}{\mu\sqrt{g}}\left(\partial_1 \left(-\gamma\frac{\partial_1^2 h}{\left(1+(\partial_1 h)^2\right)^{3/2}}\right) +\rho G \partial_1 h\right),  \qquad&&\text{on}\quad \Gamma(t)\times[0,T]\,,\\
\pat h &= \partial_1 \psi \pare{\big. x_1,h\pare{ x_1,t},t}\qquad &&\text{on }\Gamma(t)\times[0,T],
\end{alignat}
\end{subequations}

\subsection{Nondimensional Eulerian formulation} \label{sec:nondim_Darcy}
We denote by $H$ and $L$ the typical amplitude and wavelength of the interfaces in a porous medium. We change to dimensionless variables (denoted with $\tilde{\cdot}$)
\begin{align}\label{dimensionless1}
x=L \ \tilde{x}, && t=\frac{\mu L}{\rho \kappa G} \ \tilde{t},
\end{align}
and unknowns
\begin{align}\label{dimensionless2}
h(x_1,t)=H \ \tilde{h}(\tilde{x}_1,\tilde{t}), && \psi(x_1,x_2,t)=\frac{L \kappa \rho G }{\mu} \ \frac{H}{L}\  \tilde{\psi}(\tilde{x}_1,\tilde{x}_2,\tilde{t}).
\end{align}
Then,
\begin{align*}
\partial^j_{x_1} h(x_1,t)& = \frac{H}{L^j}\partial^j_{\tilde{x}_1} \tilde{h}(\tilde{x}_1,\tilde{t}), & j\in\bN,\\
%----------------------
\nabla_{x} \psi(x_1,x_2,t)& = \frac{\kappa \rho G }{\mu}\ \frac{H}{L}\ \nabla_{\tilde{x}} \tilde{\psi}(\tilde{x}_1,\tilde{x}_2,\tilde{t})
\end{align*}

\begin{alignat*}{2}
\Delta_{\tilde{x}} \tilde{\psi}&=0,  \qquad&&\text{in}\quad \widetilde{\Omega}(t)\times[0,T]\,,\\
%---------------------------------------
\nabla_{\tilde{x}_1} \tilde{\psi} \cdot {\left(-\frac{H}{L}\partial_{\tilde{x}_1} \tilde{h},1\right)}
%---------------------------------------
&=\partial_{\tilde{x}_1} \left(-\frac{\gamma}{\rho  GL^2}\frac{\partial_{\tilde{x}_1}^2 \tilde{h}}{\left(1+\left(\frac{H}{L}\partial_{\tilde{x}_1} \tilde{h}\right)^2\right)^{3/2}}\right) +  \partial_{\tilde{x}_1} \tilde{h},  \qquad&&\text{on}\quad \widetilde{\Gamma}(t)\times[0,T]\,,\\
%---------------------------------------
\partial_{\tilde{t}} \tilde{h} &= \partial_{\tilde{x}_1} \tilde{\psi} \left(\tilde{x}_1,\frac{H}{L}\tilde{h}(\tilde{x}_1,\tilde{t}),\tilde{t}\right)\qquad &&\text{on}\quad\widetilde{\Gamma}(t)\times[0,T],
\end{alignat*}
with the non-dimensionalized fluid domain
\begin{align*}
\widetilde{\Omega}(t) & =\set{ \pare{ \tilde{x}_1, \tilde{x}_2} \ \left| \ -\pi<\tilde{x}_1< \pi\,, -\infty < \tilde{x}_2 < \frac{H}{L}\tilde{h}(\tilde{x}_1,t)\,, \ t\in[0,T]\right.  }, \\
%--------------------------------
\widetilde{\Gamma}(t) & =\set{ \pare{ \tilde{x}_1, \frac{H}{L}\tilde{h}(\tilde{x}_1,t)}\,, \ t\in[0,T] }
\end{align*}
Based on our nondimensionalization of the equations, we find two dimensionless quantities of interest:
\begin{align}\label{eq:dimensionless_numbers}
\sigma=\frac{H}{L}, && \nu=\frac{\gamma}{L^2\rho G}.
\end{align}
{The Bond number} $\nu$ is a parameter that measures the ratio between the gravitational forces $L^2\rho G $ and the capillarity forces $\gamma $ and {the \emph{steepness parameter}} $\sigma$ measures the ratio between the amplitude and the wavelength of the wave. 

Dropping the tildes for the sake of clarity, we have the following dimensionless form of the Muskat problem
\begin{subequations}\label{muskat2v2}
\begin{alignat}{2}
\Delta \psi&=0,  \qquad&&\text{in}\quad \Omega(t)\times[0,T]\,,\\
\nabla \psi \cdot {\left(-\sigma\partial_1 h,1\right)}
&= {\partial_1 \left(-\nu\frac{ \partial_1^2 h}{\left(1+\left(\sigma \partial_1 h\right)^2\right)^{3/2}}\right) + \partial_1 h},  \qquad&&\text{on}\quad \Gamma(t)\times[0,T]\,,\\
\pat h &= \partial_1 \psi \left(x_1,\sigma h(x_1,t),t\right)\qquad &&\text{on}\quad \Gamma(t)\times[0,T],
\end{alignat}
\end{subequations}

\subsection{The equations in the Arbitrary Lagrangian-Eulerian formulation}\label{sec:muskat_fixed domain}
We define the time-dependent diffeomorphism
\begin{equation}\label{eq:diff_Psi}
\begin{aligned}
&\Psi : && \Omega && \rightarrow && \Omega\pare{t}\\
& && \pare{x_1, x_2} && \mapsto &&\Psi(x_1,x_2,t) = \pare{\Psi_1(x_1,x_2,t), \Psi_2(x_1,x_2,t)}=(x_1,x_2+\sigma h(x_1,t)).
\end{aligned}
\end{equation}
This diffeomorphism maps the reference domain $\Omega$ onto the moving domain $\Omega(t)$. We have hence
\begin{align}\label{eq:matrix_A}
\nabla\Psi=\left(\begin{array}{cc}1 & 0\\ \sigma\partial_1 h(x_1,t)& 1\end{array}\right), &&
A=(\nabla \Psi)^{-1}=\left(\begin{array}{cc}1 & 0\\ -\sigma\partial_1 h(x_1,t)& 1\end{array}\right).
\end{align}
With such back-to-label map defined we can now define the following new unknowns;
\begin{align}\label{eq:vort_and_stream_in_fixed_domain}
\varpi=\omega\circ\Psi, && \varphi=\psi\circ\Psi,
\end{align}
which are now defined on the fixed domain $ \Omega\times\bR^+ $. \\

Let us now remark that given any $ f\in\cC^1\pare{\Omega\pare{t}} $ the function $ f\circ \Psi \in\cC^1\pare{\Omega} $, and moreover
\begin{align*}
\partial_j \bra{f\pare{\Psi}} = \partial_k f\pare{\Psi} \  \partial_j \Psi_k &&\Longrightarrow && \nabla \bra{f\circ\Psi}=\nabla\Psi^\intercal \ \nabla f\circ\Psi, 
\end{align*}
from which we deduce that 
\begin{align}\label{eq:rule_derivation_fixed_domain}
\nabla f \circ\Psi = A^\intercal \  \nabla \bra{f\circ\Psi} &&\Longrightarrow &&   \partial_i f\circ\Psi=A^k_i \partial_k \bra{f\circ\Psi}. 
\end{align}
Similarly, we observe that
$$
\div \;v\circ\Psi= A^k_j\partial_k(v\cdot e_j)
$$
It is now easy to deduce the equation satisfied by $ \varphi = \psi\circ\Psi $ in $ \Omega\times\bR $, in fact
\begin{equation}
\label{eq:laplace_fixed_domain}
\begin{aligned}
0 & = \Delta \psi \pare{\Psi}\\
& =\div \ \nabla \psi\pare{\Psi} \\
%& = \div \ \nabla^\Psi \psi \\
%& = \div \pare{A : \nabla \varphi}\\
& = A^i_j\partial_i \pare{A^k_j \partial_k \varphi}\\
& = \partial_i\pare {A^i_j A^k_j \partial_k \varphi}. 
\end{aligned}
\end{equation}

%\noindent Let us remark that since $ h=h\pare{x_1}, \ g=g\pare{h} $ and $  \Psi_1\pare{ x_1, x_2, t} = x_1 $ we deduce that
%\begin{align*}
%h\pare{\Psi} = h\pare{\Psi_1} = h, && g\pare{\Psi}=g. 
%\end{align*}
%In the same fashion let us remark that, being the normal vector field $ n $ defined as
%\begin{equation}\label{eq:normal_vector_fixed_domain}
%n = \frac{1}{\sqrt{1+\pare{\sigma\partial_1 h}^2}}\ \pare{-\sigma \partial_1 h, 1}, 
%\end{equation}
%we deduce that $ n\pare{\Psi}=n $. 
%\noindent

In these new variables, and using $\sqrt{g} n_i=A^2_i=A^j_i N^j$, \eqref{muskat2} reads
\begin{subequations}\label{muskatALE}
\begin{alignat}{2}
{\partial_i\pare {A^i_j A^k_j \partial_k \varphi}}&=0,  \qquad&&\text{in}\quad \Omega\times[0,T]\,,\\
{A^k_j \partial_k\varphi \ A^i_j N^i }&=  { \partial_1 \left(-\nu\frac{ \partial_1^2 h}{\left(1+\pare{ \sigma\partial_1 h}^2\right)^{3/2}}\right) +  \partial_1 h},  \qquad&&\text{on}\quad \Gamma\times[0,T]\,,\\
\pat h &= \partial_1 \varphi \qquad &&\text{on}\quad\Gamma\times[0,T].
\end{alignat}
\end{subequations}

\section{The asymptotic model fortwo dimensional Darcy flow} \label{ref:asymDarcy}

A straightforward computation shows, with the help of \eqref{eq:matrix_A};
\begin{equation*}
\begin{aligned}
{\partial_i\pare {A^i_j A^k_j \partial_k \varphi}}& = \Delta \varphi - \sigma\pare{\partial_1^2 h \ \partial_2 \varphi + {2}\partial_1 h \ \partial_{12}\varphi}+ {\sigma^2(\partial_1 h)^2\partial^2_2\varphi}. 
\end{aligned}
\end{equation*}

{Similarly, using the relation \eqref{eq:rule_derivation_fixed_domain}}, we can compute
\begin{align*}
\nabla\psi\pare{\Psi}\cdot \sqrt{g} n & = {A^k_j \partial_k\varphi \ A^i_j N^i} , \\
& =  \pare{\Big. {-\sigma \partial_1 h \partial_1\varphi +(1+\sigma^2(\partial_1 h)^2)\partial_2 \varphi}}
\end{align*}

Expanding \eqref{muskatALE}, we find that
\begin{alignat*}{2}
\Delta \varphi &={\sigma\pare{\partial_1^2 h \ \partial_2 \varphi + 2\partial_1 h \ \partial_{12}\varphi}-\sigma^2(\partial_1 h)^2\partial^2_2\varphi}   \qquad&&\text{in}\quad \Omega\times[0,T]\,,\\
%------------------------------------------
 {\Big. \partial_2 \varphi}&= {\sigma \partial_1 h \partial_1\varphi -\sigma^2(\partial_1 h)^2\partial_2 \varphi}-\partial_1 \left(\nu\frac{ \partial_1^2 h}{\left(1+\pare{ \sigma\partial_1 h}^2\right)^{3/2}}\right) +  \partial_1 h,  \qquad&&\text{on}\quad \Gamma\times[0,T]\,,\\
 %------------------------------------------
\pat h &= \partial_1 \varphi \qquad &&\text{on}\quad\Gamma\times[0,T],
\end{alignat*}
Further computing the surface tension term we {obtain} that
\begin{align*}
\partial_1 \left(\nu\frac{ \partial_1^2 h}{\left(1+(\sigma\partial_1 h)^2\right)^{3/2}}\right)&=\nu\frac{ \partial_1^3 h}{\left(1+(\sigma\partial_1 h)^2\right)^{3/2}}-3\frac{\nu\sigma^2 \partial_1^2 h}{\left(1+(\sigma\partial_1 h)^2\right)^{5/2}}\partial_1 h\partial_1^2 h.
\end{align*}
{As a consequence, we have to study} the following system:
\begin{subequations}\label{muskatALE2}
\begin{alignat}{2}
\Delta \varphi &= {\sigma\pare{\partial_1^2 h \ \partial_2 \varphi + 2\partial_1 h \ \partial_{12}\varphi}-\sigma^2(\partial_1 h)^2\partial^2_2\varphi}  \;&&\text{in } \Omega\times[0,T]\,,\\
%------------------------------
 \partial_2 \varphi &= {\sigma \partial_1 h \partial_1\varphi -\sigma^2(\partial_1 h)^2\partial_2 \varphi} -\frac{\nu\ \partial_1^3 h}{\left(1+\pare{\sigma\partial_1 h}^{2}\right)^{3/2}}\nonumber\\
 &\quad+3\frac{\nu\sigma^2 (\partial_1^2 h)^2 \partial_1 h}{\left(1+\pare{\sigma\partial_1 h}^{2}\right)^{5/2}} +  \partial_1 h,  \;&&\text{on }\Gamma\times[0,T]\,,\label{muskatALE2b}\\
 %------------------------------
\pat h &= \partial_1 \varphi \; &&\text{on }\Gamma\times[0,T],
\end{alignat}
\end{subequations}

We introduce the following ansatz
\begin{align}\label{ansatz}
h(x_1,t)=\sum_{n=0}^\infty \sigma^{n}h^{(n)}(x_1,t), &&
 \varphi(x_1,x_2,t)=\sum_{n=0}^\infty \sigma^{n}\varphi^{(n)}(x_1,x_2,t).
\end{align}

Moreover since 
\begin{align*}
\frac{1}{\pare{ 1+x^2}^{3/2}} = 1+\cO\pare{x^2}, &&
\frac{1}{\pare{1+x^2}^{5/2}} = 1 +\cO\pare{x^2}, 
\end{align*}
we can rewrite \eqref{muskatALE2b} as
\begin{equation*}
\partial_2 \varphi = \partial_1 \pare{ h-\nu \partial_1^2 h} +  {  \sigma \  \partial_1 h \partial_1\varphi} + \cO\pare{\sigma^2}. 
\end{equation*}

We observe that \eqref{muskatALE2} can be written as
\begin{alignat*}{2}
\Delta \varphi &= \sigma\pare{\partial_1^2 h \ \partial_2 \varphi +  {2}\partial_1 h \ \partial_{12}\varphi} {+ \cO\pare{\sigma^2}}     \qquad&&\text{in}\quad \Omega\times[0,T]\,,\\
%-------------------------------------
 \partial_2 \varphi &= \partial_1 \pare{ h-\nu \partial_1^2 h} +  {\sigma \  \partial_1 h \partial_1\varphi} + \cO\pare{\sigma^2}
 ,  \qquad&&\text{on}\quad \Gamma\times[0,T]\,,\\
 %-------------------------------------
\pat h &= \partial_1 \varphi \qquad &&\text{on}\quad\Gamma\times[0,T],
\end{alignat*}
where $\mathcal{O}(\sigma^2)$ denotes terms of order $\sigma^2$ and higher. We are interested in finding an asymptotic model of the free boundary Darcy flow with an error $\mathcal{O}(\sigma^2)$. As a consequence, we can neglect terms of $O(\sigma^2)$ in \eqref{muskatALE2}. 
Thus, up to $\mathcal{O}(\sigma^2)$, \eqref{muskatALE2} is equivalent to
\begin{subequations}\label{muskatALE3}
\begin{alignat}{2}
\Delta \varphi &= \sigma\pare{\partial_1^2 h \ \partial_2 \varphi +  {2}\partial_1 h \ \partial_{12}\varphi}      \qquad&&\text{in}\quad \Omega\times[0,T]\,,\\
%--------------------------------
 \partial_2\varphi&=\partial_1 \pare{ h-\nu \partial_1^2 h} +  { \sigma \  \partial_1 h \partial_1\varphi} ,  \qquad&&\text{on}\quad \Gamma\times[0,T]\,,\\
\pat h &= \partial_1 \varphi \qquad &&\text{on}\quad\Gamma\times[0,T],
\end{alignat}
\end{subequations}

In order a function satisfying the ansatz \eqref{ansatz} could be a solution of \eqref{muskatALE3}, we have that each term in the asymptotic expansion has to be defined as the solution of
\begin{subequations}\label{muskatrecursion}
\begin{align}
\Delta \varphi^{(n)} &=  \sum_{j=0}^{n-1}\partial_1^2 h^{(j)} \partial_{2}\varphi ^{(n-1-j)}+ {2}\sum_{j=0}^{n-1}\partial_{1}h^{(j)} \partial_{12}\varphi^{(n-1-j)}
%&\quad - \sum_{j=0}^{n-2}\sum_{i=0}^{j}\partial_1 h^{(i)}\partial_1 h^{(j-i)}\partial_{22}\varphi^{(n-2-j)}
\text{ on }\quad \Omega\times[0,T],\\
%------------------------------------
 \partial_2\varphi^{(n)}&= \partial_1 \pare{h^{\pare{n}} - \nu\partial_1^2 h^{\pare{n}}} {+}\sum_{j=0}^{n-1} \partial_1 h^{\pare{j}} \ \partial_1 \varphi^{\pare{n-1-j}}\text{ on }\quad \Gamma\times[0,T],\\
\pat h^{(n)} &= \partial_1 \varphi^{(n)} \text{ on }\quad \Gamma\times[0,T].
\end{align}
\end{subequations}

The initial data can be assigned as 
\begin{subequations}\label{initial}
\begin{align}
h^{(0)}(x_1,0) &= h(x_1,0),\\
h^{(k)}(x_1,0) &= 0\;\forall\,k\,\geq1.
\end{align}
\end{subequations}

In particular, the terms $h^{(j)},\varphi^{(j)}$ for $j=0$ and $1$ solve

\begin{subequations}\label{case0}
\begin{align}
\Delta \varphi^{(0)} &= 0, \text{ on }\quad \Omega\times[0,T]\\
\partial_2\varphi^{(0)}&=\partial_1 \pare{h^{\pare{0}} - \nu\partial_1^2 h^{\pare{0}}} \text{ on }\quad \Gamma\times[0,T],\\
\pat h^{(0)} &= \partial_1 \varphi^{(0)}\text{ on }\quad \Gamma\times[0,T],
\end{align}
\end{subequations}
and
\begin{subequations}\label{case1}
\begin{align}
\Delta \varphi^{(1)} &=  \partial_1^2 h^{(0)} \partial_{ {2}}\varphi ^{(0)}+ 2\partial_{1}h^{(0)} \partial_{12}\varphi^{(0)} \text{ on }\quad \Omega\times[0,T]\\
%---------------------------------
\partial_2\varphi^{(1)}&=\partial_1 \pare{h^{\pare{1}} - \nu\partial_1^2 h^{\pare{1}}}  {+} \partial_1 h^{\pare{0}} \partial_1 \varphi^{\pare{0}}\text{ on }\quad \Gamma\times[0,T],\\
\pat h^{(1)} &= \partial_1 \varphi^{(1)}\text{ on }\quad \Gamma\times[0,T].
\end{align}
\end{subequations}

We observe that the solvability conditions are satisfied for both elliptic problems. Then, the explicit solution to \eqref{case0} can be computed using Lemma \ref{lem:solutions_Poisson}
\begin{equation}\label{phi0}
\varphi^{(0)}(x_1,x_2,t)=\sum_{k\in\mathbb{Z}}\frac{1}{|k|}\left((-\nu i^3k^3+ik) \widehat{h^{(0)}}(k,t)\right)e^{ix_1 k+|k|x_2}.
\end{equation}
 Then
\begin{align*}
\partial_1\varphi^{(0)}(x_1,0,t)&=\sum_{k\in\mathbb{Z}}\frac{ik}{|k|}\left((-\nu i^3k^3+ik) \widehat{h^{(0)}}(k,t)\right)e^{ix_1 k}\\
&=-\mathcal{H}\left(-\nu \partial_1^3 h^{(0)}+\partial_1 h^{(0)}\right)\\
&=-\nu \Lambda^3 h^{(0)}-\Lambda h^{(0)}.
\end{align*}
Then, we have that $h^{(0)}$ solves the following linear problem
\begin{equation}\label{h0}
\pat h^{(0)}=-\nu \Lambda^3 h^{(0)}-\Lambda h^{(0)}.
\end{equation}
We split $\varphi^{(1)}=\varphi^{(1)}_a+\varphi^{(1)}_b$, where
\begin{subequations}\label{case1a}
\begin{align}
\Delta \varphi^{(1)}_a &=  0\text{ on }\quad \Omega\times[0,T]\\
%---------------------------------
\partial_2\varphi^{(1)}_a&=\partial_1 \pare{h^{\pare{1}} - \nu\partial_1^2 h^{\pare{1}}}\text{ on }\quad \Gamma\times[0,T],
\end{align}
\end{subequations}
and
\begin{subequations}\label{case1b}
\begin{align}
\Delta \varphi^{(1)}_b &=  \partial_1^2 h^{(0)} \partial_{{2}}\varphi ^{(0)}+ 2\partial_{1}h^{(0)} \partial_{12}\varphi^{(0)}\text{ on }\quad \Omega\times[0,T]\\
%---------------------------------
\partial_2\varphi^{(1)}_b&= \partial_1 h^{\pare{0}} \partial_1 \varphi^{\pare{0}}\text{ on }\quad \Gamma\times[0,T].
\end{align}
\end{subequations}
We recall the following Lemma,
\begin{lemma}[\cite{CGSW18}]\label{lem:p1X}
Let $h:\mathbb{S}^1\to\mathbb{R}$ and $\varphi:\Omega\to \mathbb{R}$ {denote $2\pi$-periodic functions of $x_1$, such that
$$
h(x_1) = \sum_{k\in\mathbb{Z}, k\ne 0} \widehat{h}_k e^{ikx_1}\,,\quad \varphi(x_1,x_2) =
\sum_{k,m\in\mathbb{Z}} \widehat{P}_{k,m}(x_2) e^{ikx_1 + |m|x_2} \,,
$$}
where $x_2 \mapsto \widehat{P}_{k,m}(x_2)$ is a polynomial function. If $X$ is the unique solution to \begin{equation}\label{elliptic1}
\Delta X = \partial_{2}\big[2 (\partial_1 h) (\partial_1 \varphi) + (\partial_1^2 h) \varphi \big] \ \text{ in } \ \Omega \,, \ \text{ and } \
\partial_2 X = (\partial_1 h) (\partial_1 \varphi) \ \text{on } \ {\mathbb{S}^1}\,,
\end{equation}then
\begin{equation}\label{p1X}
(\partial_1 X)(x_1,0) =  - H \big[(\partial_1 h)(\partial_1 \varphi)\big] - \sum_{k,\ell,m\in \mathbb{Z}} i\text{\rm sgn}(k) |m| (\ell^2 - k^2) \widehat{h}_{k-\ell} \sum_{j=0}^\infty \frac{(-1)^j \widehat{P}^{(j)}_{\ell,m}(0)}{(|m|+|k|)^{j+1}} e^{ikx_1},
\end{equation}
where $\widehat{P}^{(j)}_{\ell,m}(0)$ denotes $\partial_2^j\widehat{P}_{\ell,m}(x_2)$  evaluated at $x_2=0$.
Moreover,
if $\varphi$ is harmonic in $\Omega$ so that $\varphi(x_1,x_2) = \sum\limits_{k\in\mathbb{Z}} \widehat{\varphi}_k e^{ikx_1 + |k|x_2}$, then
    \begin{equation}\label{p1X2}
    \partial_1 X = - \Lambda [h \partial_1\varphi] + \partial_1 (h\Lambda \varphi) = \partial_1 \big([h,H] \partial_1\varphi\big)\quad\text{on }\ {\mathbb{S}^1}\,,
    \end{equation}
where {$[h,\mathcal{H}]f=h\mathcal{H}f-\mathcal{H}(hf)$} denotes the commutator.
\end{lemma}
Then, we have that
    \begin{equation*}
    \partial_1 \varphi^{(1)}_b =\partial_{1} \big([h^{(0)},\mathcal{H}] \partial_1\varphi^{(0)}\big)=\partial_{1} \left([h^{(0)},\mathcal{H}]\left(-\nu \Lambda^3 h^{(0)}-\Lambda h^{(0)}\right)\right).
    \end{equation*}
Thus, we have that
\begin{equation}\label{h1}
\pat h^{(1)}=-\nu \Lambda^3 h^{(1)}-\Lambda h^{(1)}+\partial_{1} \left([h^{(0)},\mathcal{H}]\left(-\nu \Lambda^3 h^{(0)}-\Lambda h^{(0)}\right)\right).
\end{equation}
We define
\begin{equation}\label{f}
f=h^{(0)}+\sigma h^{(1)}.
\end{equation}
Then, we have that
\begin{equation}\label{eqfapprox}
\pat f=-\nu \Lambda^3 f-\Lambda f+\sigma\partial_{1} \left([f,\mathcal{H}]\left(-\nu \Lambda^3 f-\Lambda f\right)\right)+\cO\pare{\sigma^2}.
\end{equation}
Consequently, in the renormalized variables $f=\sigma f$,
\begin{equation}\label{eqf}
\pat f=-\nu \Lambda^3 f-\Lambda f+\partial_{1} \left([f,\mathcal{H}]\left(-\nu \Lambda^3 f-\Lambda f\right)\right).
\end{equation}
is the desired asymptotic model for the Darcy flow. 

\begin{remark}Some equivalent ways of writing \eqref{eqf} are
\begin{align}
\pat f&=-\nu \Lambda^3 f-\Lambda f+\nu\left(\Lambda \left(f \Lambda^3 f\right)-\partial_{1} \left(f \partial_1^3 f\right)\right)+\partial_{1} \left(f \partial_1 f\right)+\Lambda \left(f \Lambda f\right)\label{eqf2}\\
&=-\nu \Lambda^3 f-\Lambda f+\nu\left([\Lambda,f] \Lambda^3 f-\partial_{1} f \partial_1^3 f\right)+\left(\partial_1 f\right)^2+[\Lambda,f]\Lambda f\label{eqf3}
\end{align}
 \end{remark}
 
\section{Forchheimer flow}\label{ref:Forchheimer}
\subsection{The equations in the Eulerian formulation}
In this section we consider the fluid domain as described in \ref{ref:domain}. When the Reynolds number of the two-dimensional flow in porous media becomes larger, a correction term has to be added to \eqref{muskat}. Then, one obtaints the so-called Forchheimer equation:
\begin{subequations}\label{forchheimer}
\begin{alignat}{2}
\beta\rho |u| u+\frac{\mu}{\kappa} u+\nabla p&=-\rho G e_2,  \qquad&&\text{in}\quad \Omega(t)\times[0,T]\,,\\
\nabla\cdot u &=0,  \qquad&&\text{in}\quad \Omega(t)\times[0,T]\,,\\
p &= -\gamma \mathcal{K}_{\Gamma(t)}\qquad &&\text{on }\Gamma(t)\times[0,T],\\
\pat h &= u\cdot(-\pax h,1)\qquad &&\text{on }\Gamma(t)\times[0,T],
\end{alignat}
\end{subequations}
where the additional Forchheimer term
$$
\beta \rho |u| u
$$
accounts for high velocity inertial effects, see \cite{Bejan}. The scalar $\beta$ denotes the Forchheimer coefficient (units of $length^{-1}$). Again, the system \eqref{forchheimer} is supplemented with the initial condition \eqref{eq:initial} for $h$.

As before, we use a formulation based on the stream function, $\psi$, and the tangential velocity, $\omega$. In particular, 
\begin{align*}
 {\frac{\mu}{\kappa}}\sqrt{g}\omega&=-{\frac{\mu}{\kappa}} u\cdot \sqrt{g}\tau\\
&=\nabla p\bigg{|}_{\Gamma(t)}\cdot\sqrt{g}\tau +\rho G \partial_1 h-\beta\rho|u|u\cdot \sqrt{g}\tau\\
&=\partial_1\left( p|_{\Gamma(t)}\right) +\rho G \partial_1 h\\
&=\partial_1 \left(-\gamma\frac{\partial_1^2 h}{\left(1+(\partial_1 h)^2\right)^{3/2}}\right) +\rho G \partial_1 h-\beta\rho|u|u\cdot \sqrt{g}\tau.
\end{align*}
Then, using
$$
\nabla^\perp\cdot\left(|\nabla^\perp \psi| \nabla^\perp \psi\right)= \av{\nabla \psi} \Delta \psi + \frac{1}{\av{\nabla \psi}} \ \bra{2 \pa_1\psi\ \pa_2  \psi \ \pa _{12}\psi + \pare{\pa_1\psi}^2\pa_1^2\psi + \pare{\pa_2\psi}^2\pa_2^2\psi}
$$
we deduce that \eqref{forchheimer} is equivalent to
\begin{subequations}\label{forchheimer2}
\begin{alignat}{2}
 \frac{\mu}{\kappa} \Delta \psi = & \ -\beta\rho  \av{\nabla \psi} \Delta \psi &&\nonumber\\ 
%-------------------------------------------------
&-\beta\rho \left( \frac{1}{\av{\nabla \psi}} \ \bra{2 \pa_1\psi\ \pa_2  \psi \ \pa _{12}\psi + \pare{\pa_1\psi}^2\pa_1^2\psi + \pare{\pa_2\psi}^2\pa_2^2\psi}\right),&&\text{in}\quad \Omega(t)\times[0,T]\,,\label{forchheimer2a}\\
%----------------------------------------------------------
 \nabla \psi \cdot n  = & \ \frac{\kappa}{\mu\sqrt{g}}\left(\partial_1 \left(-\gamma\frac{\partial_1^2 h}{\left(1+(\partial_1 h)^2\right)^{3/2}}\right) +\rho G \partial_1 h\right)  \nonumber \\
%---------------------------------------
& {-\frac{\kappa \beta\rho}{\mu}|\nabla\psi|\nabla^\perp\psi\cdot\tau}, &&\text{on}\quad \Gamma(t)\times[0,T]\,,\\
\pat h = & \   \partial_1 \psi (x_1,h(x_1,t),t) &&\text{on }\Gamma(t)\times[0,T],
\end{alignat}
\end{subequations}

\subsection{Nondimensional Eulerian formulation}
We use the same nondimensional scaling introduced in Section \ref{sec:nondim_Darcy} which we recall here for the sake of clarity; 
we denote by $H$ and $L$ the typical amplitude and wavelength of the interfaces in a porous medium and consider the dimensionless variables (denoted with $\tilde{\cdot}$) defined in \eqref{dimensionless1} and \eqref{dimensionless2}.
%\begin{align*}
%x=L \ \tilde{x}, \;\; t=\frac{\mu L}{\rho \kappa G} \ \tilde{t},\;\;
%h(x_1,t)=H \ \tilde{h}(\tilde{x}_1,\tilde{t}), \;\; \psi(x_1,x_2,t)=\frac{L \kappa \rho G }{\mu} \ \frac{H}{L}\  \tilde{\psi}(\tilde{x}_1,\tilde{x}_2,\tilde{t}).
%\end{align*}
Let us denote as 
\begin{equation}\label{eq:def_f}
{\Xi}\pare{\nabla\psi, \nabla^2\psi} =  \left(  \av{\nabla \psi} \Delta \psi + \frac{1}{\av{\nabla \psi}} \ \bra{2 \pa_1\psi\ \pa_2  \psi \ \pa _{12}\psi + \pare{\pa_1\psi}^2\pa_1^2\psi + \pare{\pa_2\psi}^2\pa_2^2\psi}\right).
\end{equation}
With such notation we can compactly re-write \eqref{forchheimer2a} as
\begin{equation*}
\frac{\mu}{\kappa}\Delta \psi =-\beta\rho\ {\Xi}\pare{\nabla\psi, \nabla^2\psi}, 
\end{equation*}
from where, using the dimensionless variables and unknowns above defined, we deduce that
\begin{align*}
\frac{\mu}{\kappa}\Delta_x \psi\pare{x, t} & = \frac{\mu}{\kappa} \frac{1}{L}\frac{\kappa \rho G}{\mu}\frac{H}{L} \ \Delta_{\tilde{x} }\tilde{\psi}\pare{\tilde{x}, \tilde{t}}, \\
%----------------------------------------------------------
{\Xi}\pare{\nabla_x \psi \pare{x, t} , \nabla^2_x \psi\pare{x, t}} & = {\frac{1}{L}\pare{\frac{\kappa \rho G}{\mu}}^2 \pare{\frac{H}{L}}^2} {\Xi}\pare{\nabla_{\tilde{x}} \tilde{\psi} \pare{\tilde{x}, \tilde{t}} , \nabla^2_{\tilde{x}} \tilde{\psi} \pare{\tilde{x}, \tilde{t}}}.
\end{align*}
As a consequence, we obtain the nondimensional form of \eqref{forchheimer2a}
\begin{equation*}
\Delta_{\tilde{x}} \tilde{\psi}\pare{\tilde{x}, \tilde{t}} = \pare{\frac{\beta \kappa^2 \rho^2 G}{\mu^2}} \pare{\frac{H}{L}}{\Xi}\pare{\nabla_{\tilde{x}} \tilde{\psi} \pare{\tilde{x}, \tilde{t}} , \nabla^2_{\tilde{x}} \tilde{\psi} \pare{\tilde{x}, \tilde{t}}}. 
\end{equation*}

Performing similar computations as in Section \ref{sec:nondim_Darcy}, we can finally write the nondimensional form of the Forchheimer model
\begin{align*}
\Delta_{\tilde{x}} \tilde{\psi}& = \pare{\frac{\beta \kappa^2 \rho^2 G}{\mu^2}} \pare{\frac{H}{L}}\ {\Xi}\pare{\nabla_{\tilde{x}} \tilde{\psi}  , \nabla^2_{\tilde{x}} \tilde{\psi} } &\text{in}\quad \widetilde{\Omega}(t)\times[0,T]\,,\\
%---------------------------------------
\nabla_{\tilde{x}_1} \tilde{\psi} \cdot {\left(-\frac{H}{L}\partial_{\tilde{x}_1} \tilde{h},1\right)}
%---------------------------------------
&=\partial_{\tilde{x}_1} \left(-\frac{\gamma}{\rho  GL^2}\frac{\partial_{\tilde{x}_1}^2 \tilde{h}}{\left(1+\left(\frac{H}{L}\partial_{\tilde{x}_1} \tilde{h}\right)^2\right)^{3/2}}\right) +  \partial_{\tilde{x}_1} \tilde{h}&&\\
&\quad -\frac{H}{L}\frac{\kappa^2 \beta\rho^2 G}{\mu^2}|\nabla\tilde{\psi}|\nabla^\perp\tilde{\psi}\cdot\left(1,\frac{H}{L} \partial_{\tilde{x}_1}\tilde{h}\right),  &\text{on}\quad \widetilde{\Gamma}(t)\times[0,T]\,,\\
%---------------------------------------
\partial_{\tilde{t}} \tilde{h} &= \partial_{\tilde{x}_1} \tilde{\psi} \left(\tilde{x}_1,\frac{H}{L}\tilde{h}(\tilde{x}_1,\tilde{t}),\tilde{t}\right) &\text{on}\quad\widetilde{\Gamma}(t)\times[0,T].
\end{align*}
Defining the dimensionless constants
\begin{align*}
\sigma=\frac{H}{L}, && \nu=\frac{\gamma}{L^2\rho G}, && \lambda = \frac{\beta \kappa^2 \rho^2 G}{\mu^2}, 
\end{align*}
and dropping the tilde notation we deduce the system
\begin{subequations}\label{forchheimer2v2}
\begin{alignat}{2}
\Delta \psi&=\lambda \sigma \ {\Xi}\pare{\nabla\psi, \nabla^2\psi},  &&\text{in}\quad \Omega(t)\times[0,T]\,,\\
\nabla \psi \cdot {\left(-\sigma\partial_1 h,1\right)}
&= {\partial_1 \left(-\nu\frac{ \partial_1^2 h}{\left(1+\left(\sigma \partial_1 h\right)^2\right)^{3/2}}\right) + \partial_1 h}&&\nonumber\\
&\qquad-\lambda \sigma |\nabla\psi|\nabla^\perp\psi\cdot\left(1,\sigma \partial_{x_1}h\right), &&\text{on}\quad \Gamma(t)\times[0,T]\,,\\
\pat h &= \partial_1 \psi \left(x_1,\sigma h(x_1,t),t\right) &&\text{on}\quad \Gamma(t)\times[0,T].
\end{alignat}
\end{subequations}

\section{The asymptotic model for Forchheimer flow}\label{sec:forchheimer_fixed domain}
In this section we want to reduce the System \eqref{forchheimer2v2} to a fixed boundary and provide an asymptotic development  in terms of the{steepness} parameter $ \sigma $ analogously to what was done in Section \ref{sec:muskat_fixed domain}. Since the change of variables and the computation we do in this section are rather similar to the ones performed in detail in Section \ref{sec:muskat_fixed domain} we will often avoid to provide a computation and refer to Section \ref{sec:muskat_fixed domain} instead. 

Let us consider the $ \cC^1 $ diffeomorphism $ \Psi $ as in \eqref{eq:diff_Psi} which maps a moving {domain} $ \Omega\pare{t} $ onto the reference {domain} $ \Omega $ and let us define $ A=\pare{\nabla\Psi}^{-1} $ as in \eqref{eq:matrix_A}. Analogously as what was defined in \eqref{eq:vort_and_stream_in_fixed_domain} we define $ \varpi $ and $ \varphi $ as the back-to-label map of the vorticity and stream function respectively. We remark that, in this setting, the only new term appearing in the system \eqref{forchheimer2v2} is $ \Xi \pare{\nabla\psi, \nabla^2 \psi} $.

 Using \eqref{eq:rule_derivation_fixed_domain} we deduce that
\begin{equation*}
(\nabla\psi) \circ\Psi  =
\pare{
\begin{array}{c}
A_1^k\partial_k\varphi \\[2mm] A_2^k\partial_k\varphi
\end{array}
}=
 \pare{
\begin{array}{c}
\partial_1 \varphi \\[2mm] -\sigma \partial_1 h \partial_1\varphi +\partial_2 \varphi
\end{array}
}, 
\end{equation*}
from which we can easily obtain the following identity
\begin{equation*}
\begin{aligned}
\av{(\nabla\psi) \circ\Psi}&= \sqrt{\av{\nabla\varphi}^2 - 2\sigma \partial_1h \ \partial_1\varphi\partial_2\varphi + \sigma^2 \pare{\partial_1 h}^2\pare{\partial_1\varphi}^2}. 
\end{aligned}
\end{equation*}
Since $ \sqrt{1+x}=1+\frac{x}{2}+\cO\pare{x^2} $ we immediately deduce that
\begin{equation*}
\av{(\nabla\psi) \circ\Psi} = \av{\nabla\varphi}\left(1- \sigma \partial_1h \frac{\partial_1\varphi\partial_2\varphi}{\av{\nabla\varphi}^2}+O\pare{\sigma^2}\right). 
\end{equation*}

%We recall the relation \eqref{eq:laplace_fixed_domain} and the fact that $ h=h\pare{x_1} $ we can compute 
%\begin{equation*}
%\begin{aligned}
%& \Delta\psi \pare{\Psi} && = \partial_i\pare {A^i_j A^k_j \partial_k \varphi} \\
%& && = \Delta \varphi - \sigma\pare{\partial_1^2 h \ \partial_2 \varphi + {2}\partial_1 h \ \partial_{12}\varphi}+ {\sigma^2(\partial_1 h)^2\partial^2_2\varphi}, \\
%& && = \Delta \varphi - \sigma\pare{\partial_1^2 h \ \partial_2 \varphi + {2}\partial_1 h \ \partial_{12}\varphi}+ \cO\pare{\sigma^2}.
%\end{aligned} 
%\end{equation*}
In the same way as above, using the fact that $ \frac{1}{\sqrt{1+x}}=1-\frac{x}{2}+\cO\pare{x^2} $ we deduce
\begin{equation*}
\frac{1}{\av{(\nabla\psi) \circ\Psi}} = \frac{1}{\av{\nabla\varphi}}\left(1+\sigma \partial_1h \frac{\partial_1\varphi\partial_2\varphi}{\av{\nabla\varphi}^2}\right) +\cO\pare{\sigma^2}
\end{equation*}
Using the above observations, we can compute the leading asymptotic term of $ \sigma\lambda \ \Xi\pare{\nabla\psi, \nabla^2\psi} $:

\begin{align*}
\sigma\lambda \ \Xi\pare{\nabla\psi,\nabla^2\psi} &= {\sigma\lambda}\left(\av{\nabla\varphi}\Delta\varphi +\frac{2 \pa_1\varphi\ \pa_2  \varphi \ \pa _{12}\varphi + \pare{\pa_1\varphi}^2\pa_1^2\varphi + \pare{\pa_2\varphi}^2\pa_2^2\varphi}{\av{\nabla\varphi}}\right) + \cO\pare{\sigma^2}, \\
&=\sigma\lambda \ \Xi\pare{\nabla\varphi,\nabla^2\varphi}+\cO\pare{\sigma^2}. 
\end{align*}

Supposing now that $ \varphi $ and $ h $ admit the asymptotic expansions provided in \eqref{ansatz}, we can deduce the equations satisfied by the leading order terms $ \varphi^{\pare{0}}, \varphi^{\pare{1}}, h^{\pare{0}} $ and $ h^{\pare{1}} $ ;

\begin{subequations}\label{Fcase0}
\begin{align}
\Delta \varphi^{(0)} &= 0\text{ on }\quad \Omega\times[0,T],\\
\partial_2\varphi^{(0)}&=\partial_1 \pare{h^{\pare{0}} - \nu\partial_1^2 h^{\pare{0}}}\text{ on }\quad \Gamma\times[0,T],\\
\pat h^{(0)} &= \partial_1 \varphi^{(0)}\text{ on }\quad \Gamma\times[0,T],
\end{align}
\end{subequations}
and
\begin{subequations}\label{Fcase1}
\begin{align}
\Delta \varphi^{(1)} &=  \partial_1^2 h^{(0)} \partial_{ {2}}\varphi ^{(0)}+ 2\partial_{1}h^{(0)} \partial_{12}\varphi^{(0)}{+\lambda\av{\nabla\varphi}^{\pare{0}}\Delta\varphi^{\pare{0}}} \\
& \ + \frac{\lambda}{\av{\nabla\varphi^{\pare{0}}}} \bigg{[}2 \partial_1 \varphi^{\pare{0}} \ \partial_2 \varphi^{\pare{0}}\ \partial_{12}^2 \varphi^{\pare{0}}\nonumber\\
&\quad+ \partial_1^2 \varphi^{\pare{0}} \pare{\partial_1 \varphi^{\pare{0}}}^2 + \partial_2^2 \varphi^{\pare{0}} \pare{\partial_2 \varphi^{\pare{0}}}^2\bigg{]}\text{ on }\quad \Omega\times[0,T] \nonumber\\
%---------------------------------
\partial_2\varphi^{(1)}&=\partial_1 \pare{h^{\pare{1}} - \nu\partial_1^2 h^{\pare{1}}}  {+} \partial_1 h^{\pare{0}} \partial_1 \varphi^{\pare{0}} +\lambda |\nabla\varphi^{\pare{0}}|\partial_2\varphi^{\pare{0}} \text{ on }\quad \Gamma\times[0,T]\\
%---------------------------------------------
\pat h^{(1)} &= \partial_1 \varphi^{(1)}\text{ on }\quad \Gamma\times[0,T].
\end{align}
\end{subequations}

Then, we can compute $ \varphi^{\pare{0}} $ in terms of $ h^{\pare{0}} $, which is the solution of the equation
\begin{equation*}
\partial_t h^{\pare{0}} = -\nu \Lambda^3 h^{\pare{0}} -\Lambda h^{\pare{0}}. 
\end{equation*}

As in the Darcy flow case we can decompose $ \varphi^{\pare{1}}=\varphi^{\pare{1}}_a + \varphi^{\pare{1}}_b + \varphi^{\pare{1}}_c $ which solve 
\begin{equation*}
\begin{aligned}
\Delta \varphi^{(1)}_a &=0\text{ on }\quad \Omega\times[0,T]\\
%---------------------------------
\partial_2\varphi^{(1)}_a &=\partial_1 \pare{h^{\pare{1}} - \nu\partial_1^2 h^{\pare{1}}}  \text{ on }\quad \Gamma\times[0,T],
\end{aligned}
\end{equation*}
\begin{equation*}
\begin{aligned}
\Delta \varphi^{(1)}_b &=  \partial_1^2 h^{(0)} \partial_{ {2}}\varphi ^{(0)}+ 2\partial_{1}h^{(0)} \partial_{12}\varphi^{(0)} \text{ on }\quad \Omega\times[0,T]\\
\partial_2\varphi^{(1)}_b&=\partial_1 h^{\pare{0}} \partial_1 \varphi^{\pare{0}}\text{ on }\quad \Gamma\times[0,T],
\end{aligned}
\end{equation*}
\begin{equation*}
\begin{aligned}
\Delta \varphi^{(1)}_c &=
 \frac{\lambda}{\av{\nabla\varphi^{\pare{0}}}} \bra{2 \partial_1 \varphi^{\pare{0}} \ \partial_2 \varphi^{\pare{0}}\ \partial_{12}^2 \varphi^{\pare{0}}+ \partial_1^2 \varphi^{\pare{0}} \pare{\partial_1 \varphi^{\pare{0}}}^2 + \partial_2^2 \varphi^{\pare{0}} \pare{\partial_2 \varphi^{\pare{0}}}^2} \\
%----------------------------------------------
& = \lambda \ \Xi \pare{\nabla\varphi^{\pare{0}}, \nabla^2\varphi^{\pare{0}}} \quad
\text{ on }\quad \Omega\times[0,T] \nonumber\\
%---------------------------------
\partial_2\varphi^{(1)}_c&=\lambda |\nabla\varphi^{\pare{0}}|\partial_2\varphi^{\pare{0}}\text{ on }\quad \Gamma\times[0,T].
\end{aligned}
\end{equation*}
We note that
\begin{align*}
I&=\int_{\Omega}\frac{1}{\av{\nabla\varphi^{\pare{0}}}} \bra{2 \partial_1 \varphi^{\pare{0}} \ \partial_2 \varphi^{\pare{0}}\ \partial_{12}^2 \varphi^{\pare{0}}+ \partial_1^2 \varphi^{\pare{0}} \pare{\partial_1 \varphi^{\pare{0}}}^2 + \partial_2^2 \varphi^{\pare{0}} \pare{\partial_2 \varphi^{\pare{0}}}^2}\dx_1\dx_2\\
%-----------------------------------
&=\int_{\Omega}\nabla\cdot\left(| \nabla \varphi^{\pare{0}}|\nabla \varphi^{\pare{0}}\right)\dx _1 \dx _2\\
%------------------------------------------------
&=\int_{\Gamma}| \nabla \varphi^{\pare{0}}|\partial_2 \varphi^{\pare{0}}\dx_1,
\end{align*}
where we have used that $\varphi^{\pare{0}}$ is harmonic and the divergence theorem. As a consequence, the compatibility conditions are satisfied and the previous elliptic problems have unique solutions.

Define $f$ as in \eqref{f}. We introduce the following auxiliary function
$$
\varphi^{\aux}_1=\varphi^{(0)}+\sigma\varphi^{(1)}_a.
$$
We observe that $\varphi_{\aux}$ solves the problem
\begin{subequations}\label{varphiaux1}
\begin{align}
\Delta \varphi^{\aux}_1 &= 0\text{ on }\quad \Omega\times[0,T],\\
\partial_2 \varphi^{\aux}_1&=\partial_1 \pare{f - \nu\partial_1^2 f}\text{ on }\quad \Gamma\times[0,T].
\end{align}
\end{subequations}
We also define 
\begin{subequations}\label{varphiaux2}
\begin{align}
& \begin{multlined}
 \Delta \varphi^{\aux}_2 =\frac{\lambda}{\av{\nabla\varphi^{\aux}_1}} \bigg{[}2 \partial_1 \varphi^{\aux}_1 \ \partial_2 \varphi^{\aux}_1\ \partial_{12}^2 \varphi^{\aux}_1+ \partial_1^2 \varphi^{\aux}_1 \pare{\partial_1 \varphi^{\aux}_1}^2
 + \partial_2^2 \varphi^{\aux}_1 \pare{\partial_2 \varphi^{\aux}_1}^2\bigg{]} \\
%------------------------------- 
= \lambda \ \Xi \pare{\nabla \varphi^{\aux}_1, \nabla^2 \varphi^{\aux}_1 }\quad  \text{ on }\quad \Omega\times[0,T],
\end{multlined}\\
%----------------------------------------
& \partial_2 \varphi^{\aux}_2=\lambda |\nabla\varphi^{\aux}_1|\partial_2\varphi^{\aux}_1 \quad \text{ on }\quad \Gamma\times[0,T].
\end{align}
\end{subequations}
With this definition, and using 
\begin{align*}
\varphi^{\aux}_1-\varphi^{(0)}=\mathcal{O}(\sigma), && \Longrightarrow && \Xi \pare{\nabla \varphi^{\aux}_1, \nabla^2 \varphi^{\aux}_1 } - \Xi \pare{\nabla \varphi^{\pare{0}}, \nabla^2\varphi^{\pare{0}} } = \cO\pare{\sigma}  
\end{align*}
we find that
\begin{subequations}\label{error}
\begin{align}
\Delta \left(\varphi^{\aux}_2-\varphi^{(1)}_c\right) &= \mathcal{O}(\sigma)\text{ on }\quad \Omega\times[0,T],\\
\partial_2 \left(\varphi^{\aux}_2-\varphi^{(1)}_c\right)&=\mathcal{O}(\sigma)\text{ on }\quad \Gamma\times[0,T],
\end{align}
\end{subequations}
and, as a consequence,
$$
\sigma \varphi^{(1)}_c =\sigma \varphi^{\aux}_2+\mathcal{O}(\sigma^2).
$$
Thus, we find the following equation
\begin{align}
\partial_t f&=\partial_1 \varphi^{(0)}(x_1,0)+\sigma\partial_1 \varphi^{(1)}_a(x_1,0)+\sigma\partial_1 \varphi^{(1)}_b(x_1,0)+\sigma \partial_1 \varphi^{(1)}_c(x_1,0)\nonumber\\
&=-\nu \Lambda^3 f-\Lambda f+\sigma\partial_{1} \left([h^{(0)},\mathcal{H}]\left(-\nu \Lambda^3 h^{(0)}-\Lambda h^{(0)}\right)\right)+\sigma \partial_1\varphi^{\aux}_2(x_1,0)+\mathcal{O}(\sigma^2)\nonumber\\
&=-\nu \Lambda^3 f-\Lambda f+\sigma\partial_{1} \left([f,\mathcal{H}]\left(-\nu \Lambda^3 f-\Lambda f\right)\right)+\sigma \partial_1\varphi^{\aux}_2(x_1,0)+\mathcal{O}(\sigma^2).\label{eqfForch}
\end{align}
Finally, if we truncate at order $\mathcal{O}(\sigma^2)$ we find the asymptotic model
\begin{align}
\partial_t f&=-\nu \Lambda^3 f-\Lambda f+\sigma\partial_{1} \left([f,\mathcal{H}]\left(-\nu \Lambda^3 f-\Lambda f\right)\right)+\sigma \partial_1\varphi^{\aux}_2(x_1,0)\label{eqf2Forch},
\end{align}
where $\varphi^{\aux}_2$ solves \eqref{varphiaux2}. In the renormalized variables $f=\sigma f$, we find the following system as a model of free boundary flow in the Forchheimer regime
\begin{subequations}\label{asymforcheimer}
\begin{align}
\partial_t f&=-\nu \Lambda^3 f-\Lambda f+\partial_{1} \left([f,\mathcal{H}]\left(-\nu \Lambda^3 f-\Lambda f\right)\right)+ \partial_1\Phi(x_1,0) \label{eq:eq_F_Forchh1}\\
\Delta \Phi &= \frac{\lambda}{\av{\nabla\Upsilon}} \bigg{[}2 \partial_1 \Upsilon \ \partial_2 \Upsilon\ \partial_{12}^2 \Upsilon+ \partial_1^2 \Upsilon \pare{\partial_1 \Upsilon}^2 + \partial_2^2 \Upsilon \pare{\partial_2 \Upsilon}^2\bigg{]}\text{ on }\quad \Omega\times[0,T],\label{eq:eq_Phi_Forchh1}\\
\partial_2 \Phi&=\lambda |\nabla\Upsilon|\partial_2\Upsilon\text{ on }\quad \Gamma\times[0,T]\\
\Delta \Upsilon &= 0\text{ on }\quad \Omega\times[0,T],\\
\partial_2 \Upsilon&=\partial_1 \pare{f - \nu\partial_1^2 f}\text{ on }\quad \Gamma\times[0,T].
\end{align}
\end{subequations}

Our aim is now to express the equation \eqref{eq:eq_F_Forchh1} in terms of $ f $ only, i.e. we want to identify a nonlinear operator $ \cT $ such that
\begin{equation*}
\partial_1 \Phi(x_1,0,t)=\mathcal{T}[f]\pare{x_1, 0, t}.
\end{equation*}

Using Lemma \ref{lem:solutions_Poisson} we can compute the explicit value of $ \Upsilon $, which is
\begin{equation}\label{eq:Upsilon}
\Upsilon =\sum_{k\in\mathbb{Z}}\frac{1}{|k|}\left((-\nu i^3k^3+ik) \widehat{f}(k,t)\right)e^{ix_1 k+|k|x_2}. 
\end{equation}
 {We have} that
\begin{equation*}
\pare{\hat{\Upsilon}\pare{k, x_2, t}}_k = \pare{ \frac{1}{|k|}\left((-\nu i^3k^3+ik) \widehat{f}(k,t)\right) e^{\av{k} x_2}}_k \in \ell^2\pare{\mathbb{Z}}, \ \forall \ x_2 \leqslant 0, 
\end{equation*}
and moreover $ \Upsilon $ is real analytic in $ \mathbb{S}^1\times \pare{-\infty, 0} $, with increasing analyticity strip in the $ x_1 $--direction as $ x_2\to -\infty $. Let us denote now respectively
\begin{align}\label{eq:def_bg}
b_\lambda = \lambda \div \pare{\av{\nabla\Upsilon}\nabla \Upsilon} , && g_\lambda = \lambda \av{\nabla\Upsilon}\left. \Big. \partial_2\Upsilon\right|_{x_2=0}, 
\end{align}
and  {define}
\begin{equation}\label{eq:def_B}
\hat{B}_\lambda \pare{k, t} = \frac{1}{2}  {\int _{-\infty}^0 \hat{b}_\lambda\pare{k, y_2} e^{\av{k} y_2 } \textnormal{d} y_2} . 
\end{equation}
Whence applying again Lemma \ref{lem:solutions_Poisson},  we deduce 
\begin{equation}\label{eq:Phi_x_2=0}
\begin{aligned}
\partial_1\Phi\pare{x_1, 0, t} & = \frac{1}{\sqrt{2\pi}} \sum_{k} i \sgn\pare{k} \left\lbrace    {\hat{g}_\lambda\pare{k, 0, t} + 2{\hat{B}\pare{k, t}}}  
\right\rbrace e^{ik x_1} , \\
%------------------------------------------
& =  {-} \cH \pare{\left. g_\lambda \right|_{x_2=0}+B_\lambda } \pare{x_1, t}. 
\end{aligned}
\end{equation}

Thanks to the explicit formulation of $ \Upsilon $ provided in \eqref{eq:Upsilon} we  {have} that
\begin{equation}\label{eq:relation_derivatives_Upsilon}
\partial_1 \Upsilon =  {-}\cH \partial_2 \Upsilon. 
\end{equation}
 {Due to} the relation \eqref{eq:relation_derivatives_Upsilon} we can deduce the following identities;
\begin{align*}
\left. g \right|_{x_2=0} & = \lambda \av{\nabla\Upsilon}\left. \Big. \partial_2\Upsilon\right|_{x_2=0}, \\
%----------------------------------------------
& = \left. \lambda \sqrt{\pare{\cH \partial_2 \Upsilon}^2 + \pare{ \partial_2 \Upsilon}^2} \ \partial_2 \Upsilon\right|_{x_2=0}. 
\end{align*}
But by definition $ \partial_2 \Upsilon $ is the harmonic extension of $ \partial_1\pare{f-\nu\partial_1^2 f} $, whence
\begin{equation}
\label{eq:formulation_g}
\begin{aligned}
\left. g \right|_{x_2=0} & = \left. \lambda \sqrt{\pare{\cH \partial_2 \Upsilon}^2 + \pare{ \partial_2 \Upsilon}^2} \ \partial_2 \Upsilon\right|_{x_2=0}, \\
%----------------------------------------------
& =\lambda \sqrt{\pare{\cH \partial_1\pare{f-\nu\partial_1^2 f} }^2 + \pare{ \partial_1\pare{f-\nu\partial_1^2 f} }^2} \ \partial_1\pare{f-\nu\partial_1^2 f} .  
\end{aligned}
\end{equation}

Using \eqref{eq:Phi_x_2=0}
\begin{equation*}
\partial_1 \Phi \pare{x_1, 0, t} =  {-}\lambda \ \cH \pare{  \sqrt{\pare{\cH \partial_1\pare{f-\nu\partial_1^2 f} }^2 + \pare{ \partial_1\pare{f-\nu\partial_1^2 f} }^2} \ \partial_1\pare{f-\nu\partial_1^2 f} }   {-} \cH B_\lambda, 
\end{equation*}
which finally provides the complete evolution equation for $ f $; 
\begin{multline}\label{eq:asymptotic_Forchheimer}
\partial_t f =-\nu \Lambda^3 f-\Lambda f+\partial_{1} \left([f,\mathcal{H}]\left(-\nu \Lambda^3 f-\Lambda f\right)\right)\\
%-----------------------------------------------
 {-} \lambda \ \cH \pare{  \sqrt{\pare{\cH \partial_1\pare{f-\nu\partial_1^2 f} }^2 + \pare{ \partial_1\pare{f-\nu\partial_1^2 f} }^2} \ \partial_1\pare{f-\nu\partial_1^2 f} }  {-} \cH B_\lambda, 
\end{multline}
where $ B_\lambda $ and $ \Upsilon $ are respectively defined in \eqref{eq:def_B} and \eqref{eq:Upsilon}.

\section{Three dimensional Darcy flow with bottom topography} \label{sec:multidimensional_Darcy}
\subsection{The fluid domain}
The time-dependent three-dimensional finitely deep fluid domain, free surface and bottom boundary are defined as
\begin{align}\label{Omegat2D}
\Omega(t) & = \set{ (x_1, x_2,x_3)\in\bR^3 \ \Big| \ {-L}\pi <x_1,x_2<{L}\pi\,, -d< x_3< h(x_1,x_2,t)\,, \ t\in[0,T] }, \\
%----------------------------------
\label{Gammat2D}
\Gamma(t) & = \set{ \pare{x_1,x_2,h(x_1,x_2,t)}\in\bR^2 \ \Big| \ {-L}\pi <x_1,x_2<{L}\pi\,,\ t\in[0,T] }
\\
%----------------------------------
\label{Gammabott2D}
\Gamma_{\textnormal{bot}} & = \set{ \pare{x_1,x_2,-d}\in\bR^2 \ \Big| {-L}\pi <x_1,x_2<{L}\pi\,,\ t\in[0,T] }
\end{align} 
with periodic boundary conditions in the horizontal variables $x_1,x_2$. 

\subsection{The equations in the Eulerian formulation}
In this section we consider the free boundary Darcy problem in three dimensions. We assume that the domain is bounded from below by a flat bottom situated at $ x_3= -d $. The free boundary Darcy problem in such configuration reads as follows:
\begin{align*}\label{muskat}
\frac{\mu}{\kappa} u+\nabla p&=-\rho G e_3,  \qquad&&\text{in}\quad \Omega(t)\times[0,T]\,,\\
\nabla\cdot u &=0,  \qquad&&\text{in}\quad \Omega(t)\times[0,T]\,,\\
p &= -\gamma \mathcal{K}_{\Gamma(t)}\qquad &&\text{on }\Gamma(t)\times[0,T],\\
\pat h &= u\cdot \tilde{n}\qquad &&\text{on }\Gamma(t)\times[0,T], \\
u_3 & = 0, \qquad &&\text{on }\Gamma_{\textnormal{bot}}\times[0,T],
\end{align*}
where $\tilde{n}$ is the non-unitary outward pointing normal vector.

\noindent Indeed in such configuration $ \cK_{\Gamma\pare{t}} $ is the mean curvature of the surface. Since in out setting $ \Gamma\pare{t} $ is given as a graph, the mean curvature assumes the explicit form (cf. \cite{SpivakIII})
\begin{equation*}
\begin{aligned}
%\cK_{\Gamma\pare{t}} & =  \ \nabla_{X, z}\cdot \pare{\frac{\nabla_{X, z}\pare{z-h}}{\av{\nabla_{X, z}\pare{z-h}}}}, \\
%-------------------------------------------------------------
\cK_{\Gamma\pare{t}} & = \frac{\pare{1+\pare{\partial_1 h}^2} \partial_2^2 h  + \pare{1+\pare{\partial_2 h}^2} \partial_1^2 h - 2 \partial_1 h \ \partial_2 h \ \partial^2_{12} h}{\pare{1+ \pare{\partial_1 h}^2 + \pare{\partial_2 h}^2}^{3/2}}
\end{aligned}
\end{equation*}

We observe that $u=\nabla \Phi$ where the potential function is given by
\begin{equation*}
\Phi=\frac{\kappa}{\mu}\left(-p-G\rho x_3\right). 
\end{equation*}
With this notation, the equation for the free surface becomes
$$
\partial_t h=\sqrt{1+(\partial_1h)^2+(\partial_2h)^2} \ \left. \partial_n \Phi \right| _{x_3=h}.
$$
Also, using the divergence free condition for the velocity field, we have that $ \Phi $ solves the elliptic problem
\begin{equation}\label{eq:elliptic_equation_Phi}
\left\lbrace
\begin{aligned}
& \Delta \Phi =0, && \text{ in } \Omega\pare{t}, \\
& \Phi = \frac{\gamma\kappa}{\mu} \  \cK_{\Gamma\pare{t}} -\frac{\kappa\rho \ G}{\mu} \ h && \text{ on } \Gamma\pare{t}, \\
& \partial_3\Phi =0 , && \text{ on } \Gamma_{\textnormal{bot}}, 
\end{aligned}
\right. 
\end{equation}
This elliptic equation is uniquely solvable if the zero mean function $ h $ is sufficiently regular (cf. \cite[Chapter 2]{Lannes13}). We can hence completely determine $ \Phi $ from its trace, \emph{i.e.} from $h$ and its derivatives. Thus, the previous equation for the free boundary can equivalently be restated as 
\begin{equation*}
\partial_t h = \cG  \pare{\frac{\gamma\kappa}{\mu} \  \cK_{\Gamma\pare{t}} -\frac{\kappa\rho \ G}{\mu} \ h}, 
\end{equation*}
where $ \cG $ is the \textit{Dirichlet--Neumann} (DN) operator (cf. \cite[Chapter 3]{Lannes13}), \emph{i.e.} the operator that solves the elliptic problem for $\Phi$, compute its normal gradient and takes the trace of this normal gradient up to the boundary.\\

\subsection{Nondimensional Eulerian formulation}
We can now nondimensionalize our equations following the very same procedure explained in Section \ref{sec:nondim_Darcy}. We define the new variables
\begin{align}\label{dimensionless12D}
(x_1,x_2)=L  (\tilde{x_1},\tilde{x_2}),\; x_3=d \tilde{x_3}, && t=\frac{\mu L}{\rho \kappa G} \ \tilde{t},
\end{align}
and unknowns
\begin{align}\label{dimensionless22D}
h(x_1,x_2,t)=H \ \tilde{h}(\tilde{x}_1,\tilde{x}_2,\tilde{t}), && \Phi(x_1,x_2,x_3,t)=\frac{H \kappa \rho G }{\mu} \tilde{\Phi}(\tilde{x}_1,\tilde{x}_2,\tilde{x}_3,\tilde{t}).
\end{align}
We define the following non-dimensional parameters
$$
\delta=\frac{d^2}{L^2},\;\varepsilon=\frac{H}{d}.
$$
These dimensionless quantities are known in the literature as the \emph{shallowness} and \emph{amplitude} parameters. We observe that 
$$
\sigma=\varepsilon\sqrt{\delta}.
$$
The equations in nondimensional form read as follows
\begin{align*}
\delta\left(\partial_1^2\Phi+\partial_2^2\Phi\right)+\partial_3^2\Phi & =0, & \text{ in } \Omega\pare{t}\times \bra{0, T}, \\
%-----------------------------------------------
\Phi & = \nu\cK^{\sigma}_{\Gamma(t)}- h , & \text{ on } \Gamma\pare{t}\times \bra{0, T}, \\
%-----------------------------------------------
\partial_t h & = \cG\pare{\nu\cK^{\sigma}_{\Gamma(t)}- h}, & \text{ on } \Gamma\pare{t}\times \bra{0, T},  \\
%-----------------------------------------------
\partial_3 \Phi& =0 , & \text{ on } \Gamma_{\textnormal{bot}}\times \bra{0, T}.
\end{align*}
where the Bond number was given in \eqref{eq:dimensionless_numbers}, the nondimensional fluid domain and free surface are
\begin{align}\label{Omegat2Dv2}
\Omega(t) & = \set{ (\tilde{x}_1, \tilde{x}_2,\tilde{x}_3)\in\bR^3 \ \Big| \ \pi <x_1,x_2<\pi\,, -1< x_3< \varepsilon h(x_1,x_2,t)\,, \ t\in[0,T] }, \\
%----------------------------------
\label{Gammat2Dv2}
\Gamma(t) & = \set{ \pare{\tilde{x}_1,\tilde{x}_2,\varepsilon h(\tilde{x}_1,\tilde{x}_2,t)}\in\bR^2 \ \Big| \ \pi <x_1,x_2<\pi\,,\ t\in[0,T] }
\\
%----------------------------------
\label{Gammabott2Dv2}
\Gamma_{\textnormal{bot}} & = \set{ \pare{\tilde{x}_1,\tilde{x}_2,-1}\in\bR^2 \ \Big| \pi <\tilde{x}_1,\tilde{x}_2<\pi\,,\ t\in[0,T] }
\end{align} 
and the non-dimensional curvature is given by
\begin{equation*}
\cK^{\sigma}_{\Gamma(t)} =\ \frac{\pare{1+\pare{\sigma\partial_1 h}^2} \partial_2^2 h  + \pare{1+\pare{\sigma \partial_2 h}^2} \partial_1^2 h - 2\sigma^2 \partial_1 h \ \partial_2 h \ \partial^2_{12} h}{\pare{1+ \pare{\sigma\partial_1 h}^2 + \pare{\sigma\partial_2 h}^2}^{3/2}} \ , 
\end{equation*}

\section{The asymptotic model for three dimensional Darcy flow in presence of a bottom topography} \label{ref:asymDarcy3D}
As we are interested in the weak nonlinearity limit (and not in the shallow water limit), we fix now $\delta=1$ (so, $\sigma$ and $\varepsilon$ are comparable).\\

In the previous section we have obtained a closed formulation for the evolution of $ h $ (albeit it is highly nontrivial and nonlinear) in terms of the Dirichlet-Neumann operator. It is though possible to perform a development of the DN operator in terms of the steepness parameter $ \sigma $ (cf. \cite{CSS92, CSS97} and \cite[Section 3.6.2]{Lannes13}) around the rest state. Let us define the linear operator
\begin{equation*}
\widehat{\cG_0 \phi}\pare{\xi}  = \av{\xi}\tanh\av{\xi}\ \hat{\phi}\pare{\xi}, 
\end{equation*}
then following \cite{Lannes13} we know that 
\begin{equation*}
\cG \phi = \cG_0 \phi - \sigma\pare{\big. \cG_0\pare{ h \cG_0 \phi} + \nabla\cdot\pare{h\nabla\phi}} + \cO\pare{\sigma^2}, 
\end{equation*}
and since
\begin{equation*}
\cK^{\nu, \sigma}_{h\pare{t}} = \nu \Delta h + \cO\pare{\sigma^2}, 
\end{equation*}
we can drop the $ \cO\pare{\sigma^2} $ contributions to deduce the following asymptotic model 
\begin{equation}\label{eq:multid}
\partial_t h - \nu \cG_0 \Delta h + \cG_0 h = - \nu\sigma \pare{\big. \cG_0\pare{ h \cG_0 \Delta h} + \nabla\cdot\pare{h\nabla\Delta h}}  + \sigma\pare{\big. \cG_0\pare{ h \cG_0 h} + \nabla\cdot\pare{h\nabla h}}.
\end{equation}
So, in the renormalized variables $f=\sigma h$, we find that
\begin{equation}\label{eq:multid2}
\partial_t f - \nu \cG_0 \Delta f + \cG_0 f = - \nu \pare{\big. \cG_0\pare{ f\cG_0 \Delta f} + \nabla\cdot\pare{f\nabla\Delta f}}  + \pare{\big. \cG_0\pare{ f \cG_0 f} + \nabla\cdot\pare{f\nabla f}}.
\end{equation}

\begin{remark}
In the case in which there is no bottom the first-order approximation of the DN operator is 
\begin{equation*}
\cG_0 \phi = \Lambda \phi.
\end{equation*}
Hence, in the case where the depth is infinite and the flow is three dimensional, we recover the multi-dimensional asymptotic model
\begin{equation}\label{eq:multid3}
\partial_t f - \nu \Lambda\Delta f + \Lambda f = - \nu \pare{\big. \Lambda\pare{ f\Lambda \Delta f} + \nabla\cdot\pare{f\nabla\Delta f}}  + \pare{\big. \Lambda\pare{ f \Lambda f} + \nabla\cdot\pare{f\nabla f}}.
\end{equation}
This model is completely analogous to \eqref{eqf}.
\end{remark}

\appendix
\section{The explicit solution of an elliptic problem}
\begin{lemma}\label{lem:solutions_Poisson}
Let us consider the Poisson equation
\begin{equation}
\label{eq:Poisson}
\left\lbrace
\begin{aligned}
& \Delta u \pare{x_1 , x_2} &&= b \pare{x_1 , x_2}, & \pare{x_1, x_2} &\in \mathbb{S}^1\times  \pare{-\infty, 0}, \\
%-----------------------------------------------------------
& \partial_2 u \pare{x_1, 0} &&= g \pare{x_1} , & x_1 &\in \mathbb{S}^1, \\
& u \pare{x_1, -\infty} && =0, & x_1 &\in \mathbb{S}^1,
\end{aligned}
\right. 
\end{equation}
where we  {assume that the forcing $ b\in H^4(\Omega)$ and $g\in H^1(\Omega)$ satisfy the compatibility condition
$$
\int_\Omega b(x_1,x_2)dx_1dx_2=\int_\Gamma g(x_1)dx_1. 
$$
Then, }the unique solution $ u $ of \eqref{eq:Poisson}  is
\begin{equation} \label{eq:solution_Poisson}
\begin{aligned}
u \pare{x_1, x_2} =  {-} \frac{1}{\sqrt{2\pi}} \sum_{k}& \left\lbrace  \frac{1}{\av{k}} \bra{\frac{1}{2} \int _{ -\infty}^0  \hat{b}\pare{k, y_2} e^{\av{k} y_2 } \textnormal{d} y_2  {-} \hat{g}\pare{k} } e^{\av{k} x_2} \right.  \\
%------------------------------------
& +\frac{1}{2\av{k}} \int _{ -\infty}^0  \hat{b}\pare{k, y_2} e^{\av{k} y_2 } \textnormal{d} y_2
\  e^{- \av{k} x_2}  \\
%------------------------------------
& \left. 
+ \int_0^{x_2} \frac{\hat{b}\pare{k, y_2}}{2\av{k}} \bra{e^{\av{k}\pare{y_2 - x_2 }} - e^{\av{k}\pare{x_2 - y_2 }}} \textnormal{d} y_2, 
\right\rbrace e^{ik x_1} , 
\end{aligned}
\end{equation} 
where the operator $ \hat{\cdot} $ denotes the Fourier transform in the variable $ x_1 $. 
\end{lemma}

\begin{proof}
Let us apply the Fourier transform to the equation \eqref{eq:Poisson}, this transforms the PDE \eqref{eq:Poisson} in the following series of second-order inhomogeneous costant coefficients ODE's
\begin{equation}
\label{eq:Poisson_Fourier}
\left\lbrace
\begin{aligned}
& -k^2 \hat{u}\pare{k, x_2} + \partial_2^2 \hat{u}\pare{k, x_2} = \hat{b} \pare{k, x_2}, & \pare{k, x_2} & \in \mathbb{Z}\times \pare{-\infty, 0}, \\
%------------------------------------
& \partial_2 \hat{u} \pare{k, 0} = \hat{g}\pare{k}, & k & \in \mathbb{Z} , \\
%------------------------------------
& \hat{u}\pare{k, -\infty} =0 , & k & \in \mathbb{Z}. 
\end{aligned}
\right. 
\end{equation}

The generic solution of \eqref{eq:Poisson_Fourier} can be deduced using the variation of parameters method, whence 
\begin{equation*}
\hat{u}\pare{k, x_2} = C_1 \pare{k} e^{\av{k}x_2} + C_2 \pare{k} e^{- \av{k}x_2} - \int_0^{x_2} \frac{\hat{b}\pare{k, \xi_2}}{2\av{k}} \bra{e^{\av{k}\pare{y_2 - x_2 }} - e^{\av{k}\pare{x_2 - y_2 }}} \textnormal{d} y_2.
\end{equation*}

The boundary conditions determine the values of the $ C_i $'s
\begin{align*}
 C_2 \pare{k} & = - \frac{1}{2\av{k}} \int _{ -\infty}^0 \hat{b}\pare{k, y_2} e^{\av{k} y_2 } \textnormal{d} y_2, 
 &
 C_1\pare{k} & = C_2\pare{k} + \frac{\hat{g}\pare{k}}{\av{k}}.
\end{align*}
We emphasize that $ C_2\pare{k} $ is well-defined for each $ k \neq 0 $. We hence obtain the expression \eqref{eq:solution_Poisson}; in particular we can write $ u=u_++u_- $ where
\begin{align*}
& \begin{multlined}
u_-\pare{x_1, x_2}  = -\frac{1}{\sqrt{2\pi}} \sum_{k} \left\lbrace  \frac{1}{\av{k}} \bra{\frac{1}{2} \int _{ -\infty}^0  \hat{b}\pare{k, y_2} e^{\av{k} y_2 } \textnormal{d} y_2  {-} \hat{g}\pare{k} } e^{\av{k} x_2} \right. \\
 -
\left.  \int_0^{x_2} \frac{\hat{b}\pare{k, y_2}}{2\av{k}} { e^{\av{k}\pare{x_2 - y_2 }}} \textnormal{d} y_2, 
\right\rbrace e^{ik x_1} , 
\end{multlined} \\
%-----------------------------------------------------
& \begin{multlined}
u_+\pare{x_1, x_2}  =  - \frac{1}{\sqrt{2\pi}} \sum_{k} \left\lbrace
\frac{1}{2\av{k}} \int _{ -\infty}^0  \hat{b}\pare{k, y_2} e^{\av{k} y_2 } \textnormal{d} y_2
\  e^{- \av{k} x_2} \right. \\
\hspace{1cm}\left. 
+\int_0^{x_2} \frac{\hat{b}\pare{k, y_2}}{2\av{k}} {e^{\av{k}\pare{y_2 - x_2 }} } \textnormal{d} y_2, 
\right\rbrace e^{ik x_1} . 
\end{multlined}
\end{align*}
Indeed $ \av{u_-}< \infty $ since it is defined via  {the} negative exponential weights  {$e^{|k|x_2}$}, while we can reformulate $ u_+ $ as 
\begin{equation*}
u_+\pare{x_1, x_2}  =  - \frac{1}{\sqrt{2\pi}} \sum_{k} \set{
\frac{1}{2\av{k}} \int _{ -\infty}^{x_2}  \hat{b}\pare{k, y_2} e^{\av{k} y_2 } \textnormal{d} y_2}
\  e^{ik x_1- \av{k} x_2}.
\end{equation*}
 {As $y_2\leq x_2\leq0$ we have that $y_2-x_2\leq0$ and the definition of $u_+$ involves negative exponential weights. Similarly, when $k=0$, the compatibility condition for $b$ and $g$ ensures that $\hat{u}$ is well-defined concluding the proof.} 
\end{proof}

\section*{Acknowledgments}
The research of S.S. is supported by the Basque Government through the BERC 2018-2021 program and by Spanish Ministry of Economy and Competitiveness MINECO through BCAM Severo Ochoa excellence accreditation SEV-2017-0718 and through project MTM2017-82184-R funded by (AEI/FEDER, UE) and acronym "DESFLU".\\
We would like as well to express our gratitude to the unknown referee whose comments greatly improved the quality of the present work.

\begin{footnotesize}
%\bibliography{database}
%\bibliographystyle{plain}

\end{footnotesize}
\vspace{2cm}

\end{document}